\documentclass[a4paper,12pt]{article}
\usepackage{amsmath,amsthm,amssymb}
\usepackage[T2A]{fontenc}
\usepackage{amscd}

\usepackage[table]{xcolor} 

\usepackage{hyperref}

  \newskip\prethm \prethm3.0pt plus1.3pt minus.4pt
  \newskip\posthm \posthm2.7pt plus1.4pt minus.3pt
  \newtheoremstyle{STATEMENT}%
       {\prethm}{\posthm}{\itshape}{\parindent}{\scshape}
       {.}{.6em plus.2em minus.1em}{}
  \newtheoremstyle{EXPLANATION}%
       {\prethm}{\posthm}{}{\parindent}{\scshape}
       {.}{.6em plus.2em minus.1em}{}
\theoremstyle{STATEMENT}

\newtheorem{theorem}{Theorem}[section]

\newtheorem{lemma}{Lemma}[section]
\newtheorem{assertion}{Assertion}[section]
\newtheorem{corollary}{Corollary}[section]

\theoremstyle{EXPLANATION}
\newtheorem{definition}{Definition}[section]
\newtheorem{remark}{Remark}[section]

\begin{document}

\title{Local bi-integrability of bi-Hamiltonian systems, \\ Part II: Real smooth case}
\author{I.\,K.~Kozlov\thanks{No Affiliation, Moscow, Russia. E-mail: {\tt ikozlov90@gmail.com} }
}
\date{}

\maketitle

\begin{abstract} We prove that any bi-Hamiltonian system $v = \left(\mathcal{A} + \lambda \mathcal{B}\right)dH_{\lambda}$ on a real smooth manifold that is Hamiltonian with respect all Poisson brackets $\left(\mathcal{A} + \lambda \mathcal{B}\right)$ is locally bi-integrable.  We construct a complete set of functions $\mathcal{G}$ in bi-involution by extending the set of standard integrals $\mathcal{F}$ consisting of Casimir functions of Poisson brackets, eigenvalues of the Poisson pencil, and the Hamiltonians. Moreover,  we show that at a generic point of $M$ differentials of the extended family $d \mathcal{G}$ can realize any bi-Lagrangian subspace $L$ containing the differentials of the standard integrals $d \mathcal{F}$. 
\end{abstract}

\tableofcontents

\section{Introduction}

This article serves as a continuation of the research presented in \cite{Kozlov24BiIntBiPoissReduct}. Prior familiarity with that work is advisable. We use the notation and statements established in \cite{Kozlov24BiIntBiPoissReduct}.
Let $M$ be a real $C^{\infty}$-smooth manifold and $\mathcal{P} = \left\{ \mathcal{A}_{\lambda} = \mathcal{A} + \lambda \mathcal{B} \right\}_{\lambda \in \bar{\mathbb{R}}}$ be a \textbf{Poisson pencil} of compatible Poisson bracket on it. Here  $\bar{\mathbb{R}} = \mathbb{R} \cup \left\{\infty\right\}$ and $\mathcal{A}_{\infty} = \mathcal{B}$.   A dynamical system $v = \dot{x}$ on $M$ is called \textbf{bi-Hamiltonian w.r.t. a pencil} $\mathcal{P}$ if it is Hamiltonian w.r.t. all brackets of the pencil, i.e. for any $\lambda \in \bar{\mathbb{R}}$ there exists a (smooth) function $H_{\lambda}$ such that\begin{equation} \label{Eq:BiHam1} v = \mathcal{A}_{\lambda} d H_{\lambda}. \end{equation}  The \textbf{rank} of a Poisson pencil $\mathcal{P} = \left\{\mathcal{A} + \lambda \mathcal{B}\right\}$ on $M$ is \begin{equation} \label{Eq:PencilRank} \operatorname{rk} \mathcal{P}=\max_{x \in M, \lambda \in \bar{\mathbb{R}}} \left( \mathcal{A}(x) + \lambda \mathcal{B}(x) \right). \end{equation}

\begin{definition} A bi-Hamiltonian system \eqref{Eq:BiHam1} on a manifold $M$ is \textbf{locally bi-integrable} if  in a neighborhood of a generic point $x \in M$ there exist a set of integrals $\mathcal{G} = \left\{g_1, \dots, g_N\right\}$  satisfying the following conditions:

\begin{enumerate}
    \item All the functions $g_1,\dots, g_N$ are in \textbf{bi-involution}, meaning they commute w.r.t. all Poisson brackets $\mathcal{A}_{\lambda} = \mathcal{A} + \lambda \mathcal{B}, \lambda \in \bar{\mathbb{R}}$.
    
    \item The family $\mathcal{G}$ is \textbf{complete}, meaning that $N =  \dim M - \frac{1}{2}  \operatorname{rk} \mathcal{P}$ and $g_1, \dots, g_N$ are functionally independent, i.e. $dg_1 \wedge \dots \wedge dg_N \not = 0$ almost everywhere.

\end{enumerate}  \end{definition}

The following statement was proved in \cite{Kozlov24BiIntBiPoissReduct} for the real smooth case, when all eigenvalues of the Poisson pencil  $\mathcal{P} = \left\{\mathcal{A} + \lambda \mathcal{B}\right\}$ are real, and for the complex analytic case. In this paper, we extend that result to encompass the general real smooth case.

\begin{theorem} \label{T:MultHamBiInt} Let $\mathcal{P} = \left\{ \mathcal{A} + \lambda \mathcal{B}\right\}$ be a Poisson pencil on a  real $C^{\infty}$-smooth manifold $M$. If a vector field $v$ is bi-Hamiltonian w.r.t. $\mathcal{P}$, then it is locally bi-integrable. \end{theorem}

We prove the following more general Theorem~\ref{T:MainGenTh} in Section~\ref{S:MainProof}. As in \cite{Kozlov24BiIntBiPoissReduct}, the family of integrals $\mathcal{G}$ from Theorem~\ref{T:MainGenTh} contains the following family of standard integrals.

\begin{definition} \label{Def:StandInt} For an open subset $U \subset M$ the \textbf{family of standard integrals} $\mathcal{F}$ on $U$ consists of the following functions:

\begin{enumerate}
    \item Casimir functions $f_{\lambda}$ for brackets $\mathcal{A}_{\lambda}$ that are regular on $U$.

    \item Eigenvalues $\lambda_j(x)$ of the pencil $\mathcal{P}$. We replace each pair of complex-conjugate eigenvalues $\alpha_j(x) \pm i \beta_j(x)$ with the real part $\alpha_j(x)$ and the imaginary part $\beta_j(x)$. 
    
    \item The Hamiltonians $H_{\alpha}$ for all $\alpha \in \bar{\mathbb{R}}$. 
\end{enumerate}

\end{definition}

Similar to the analytic case, the standard integrals are first integrals of a bi-Hamiltonian system (see  Section~\ref{SubS:StandInt}). Casimir functions and eigenvalues may not be well-defined on the entire manifold $M$.  Therefore, we restrict our attention to a sufficiently small neighborhood\footnote{Alternatively, one could consider the germs of these local integrals.} of a point $x \in M$ to ensure their well-definedness. By \cite[Corollary 3.1]{Kozlov24BiIntBiPoissReduct} differentials of local Casimir functions of regular Poisson brackets span the \textbf{core distribution} $\mathcal{K} \subset T^*M$, given by
 \[ \mathcal{K}_x =  \bigoplus_{\lambda - \text{regular for $\mathcal{P}(x)$}} \operatorname{Ker}\mathcal{A}_\lambda (x). \]

\begin{definition} \label{Def:SmallNeigh} We say that a neighborhood $Ux$ of a point $x\in M$ is \textbf{small} if the following two conditions hold:

\begin{enumerate}
    \item The core distribution $\mathcal{K} \subseteq d \mathcal{F}$, where $\mathcal{F}$ is the family of standard integrals on $Ux$.
    
    \item All eigenvalues $\lambda_j$ are finite, i.e. $\lambda_j < \infty$, and are well-defined functions on $Ux$.
\end{enumerate} 

\end{definition}

A point $x \in (M,\mathcal{P})$ is \textbf{JK-regular} if in a neighborhood of $Ox_0$ the pencils $\mathcal{P}(x)$ have the same Kronecker Canonical Form\footnote{The Kronecker canonical form, as established in the Jordan-Kronecker theorem, is applicable to linear pencils defined on complex vector spaces. Formally, we should consider the complexifications of the Poisson pencils $\mathcal{P}(x)$.}, up to the eigenvalues\footnote{Some authors say that $\mathcal{P}(x)$ belong to the same bundle or have the same algebraic type. "Algebraic type" and "bundle of a linear pencil" is roughly the same thing.} (see \cite[Definition 3.1]{Kozlov24BiIntBiPoissReduct}). It is evident that any JK-regular point $x\in M$ that possesses finite eigenvalues $\lambda_j(x) < \infty$ has a small neighborhood. Bi-Lagrangian subspaces are defined in Definition~\ref{Def:BiLagr}. Our main result is the following. 

\begin{theorem} \label{T:MainGenTh} Let $\mathcal{P}$ be a Poisson pencil on $M$ and $v = \mathcal{A}_\lambda d H_{\lambda}$ be a vector field that is bi-Hamiltonian w.r.t. $\mathcal{P}$. Let $x_0 \in M$ be a JK-regular point and $\mathcal{F}$ be a family of standard integrals in a small neighborhood $Ux_0$. Assume that the following two conditions are satisfied:

\begin{enumerate} 

\item Locally, within a neighborhood of $x_0$, \[\dim d \mathcal{F}(x) = \operatorname{const}.\]

\item After bi-Poisson reduction w.r.t. $d \mathcal{F}$ the point $x_0$ remains JK-regular. 

\end{enumerate}

Then in a sufficiently small neighborhood $Ox_0$ the family $\mathcal{F}$ can be extended to complete family of functions $\mathcal{G} \supset \mathcal{F}$ in bi-involution.  Moreover, for any bi-Lagrangian subspace $L \subset T_{x_0}M$ such that $d\mathcal{F}(x_0) \subset L$ we can choose the extension $\mathcal{G}$ such that \begin{equation} \label{Eq:RealizeLag} d \mathcal{G}(x_0)  = L.\end{equation} \end{theorem}

This theorem can be considered as an answer to the questions \cite[Problem 13]{BolsinovTsonev17} and \cite[Problem 4.3]{BolsinovOpen17} regarding the local realization of bi-Lagrangian subspaces by the distribution $d\mathcal{G}$. These questions were initially posed in a broad and general manner. Since, it is "reasonable" to include the standard integrals $\mathcal{F}$ into the family $\mathcal{G}$, it is natural to consider only bi-Lagrangian subspaces $L$ containing $d \mathcal{F}$.

\subsection{Conventions and acknowledgements}

\textbf{Conventions.} All manifolds (functions, Poisson brackets, etc)  are  real $C^{\infty}$-smooth. Some property holds ``almost everywhere'' or ``at a generic point'' of a manifold $M$ if it holds on an open dense subset of $M$.  We denote $\bar{\mathbb{R}} = \mathbb{R} \cup \left\{ \infty \right\}$ and $\bar{\mathbb{C}} = \mathbb{C} \cup \left\{ \infty \right\}$.

\par\medskip

\textbf{Acknowledgements.} The author would like to thank A.\,V.~Bolsinov, A.\,M.~Izosimov and A.\,Yu.~Konyaev for useful comments. 

\section{Bi-Lagrangian subspaces}

Let $A$ and $B$ be skew-symmetric bilinear forms
on a finite-dimensional complex vector space $V$. We call a one-parametric family of skew-symmetric forms \[\mathcal{L} = \left\{A + \lambda B \, \, \bigr| \,\, \lambda \in \bar{\mathbb{C}} \right\}\] a \textbf{linear pencil}. 
The \textbf{rank} of a linear pencil $\mathcal{L} = \left\{ A + \lambda B\right\}$ is \[ \operatorname{rk} \mathcal{L} = \max_{\lambda \in\bar{\mathbb{C}}} \operatorname{rk} (A +\lambda B).\] A value $\lambda_0 \in \bar{\mathbb{C}}$ is \textbf{regular} if $\operatorname{rk} A_{\lambda_0} = \operatorname{rk} \mathcal{L}$. Bi-Lagrangian subspaces were extensively studied in   \cite{Kozlov24BiLagr}.

\begin{definition} \label{Def:BiLagr} A subspace $U \subset V$ of a bi-Poisson vector space $(V, \mathcal{L})$ is called 

\begin{itemize}

\item \textbf{admissible} if its skew-orthogonal complements $U^{\perp_{A_\lambda}}$ coincide for almost all forms $A_\lambda \in \mathcal{L}$.

\item \textbf{bi-isotropic} if $A_{\lambda}(u, v) = 0$ for all $u, v\in V$ and all $A_{\lambda} \in \mathcal{L}$;

\item  \textbf{bi-Lagrangian} if it is bi-isotropic and $\dim U = \dim V - \frac{1}{2} \operatorname{rk} \mathcal{L}$.

\end{itemize}

 \end{definition}

By \cite[Assertion 3.2]{Kozlov24BiLagr} a subspace $U \subset (V,\mathcal{L})$ is bi-Lagrangian if and only if it is maximal (w.r.t. inclusion) bi-isotropic and admissible.

\begin{lemma}[{\cite[Lemma 3.3]{Kozlov24BiLagr}}] \label{L:NonGenDescBiLagr} Let $\mathcal{P} = \left\{ A + \lambda B \right\}$ be a linear pencil on $V$. Assume that $B$ is nondegenerate (i.e. $\operatorname{Ker} B = 0$) and let $P = B^{-1}A$ be the recursion operator.  A subspace $L \subset (V, \mathcal{P})$ is bi-Lagrangian w.r.t. $B$ if and only if it is Lagrangian w.r.t. $B$ and $P$-invariant. \end{lemma}

We need the following simple statement.

\begin{corollary} \label{Cor:BiLagrNilp} Let $\mathcal{P} = \left\{ A + \lambda B \right\}$ be a linear pencil and $\operatorname{Ker} B = 0$. Let $P = B^{-1}A$ be the recursion operator, $N$ be its nilpotent part and $A' =  B \circ N$. If a subspace $L \subset V$ is bi-Lagrangian w.r.t. $\mathcal{P}$, then it is also bi-Lagrangian w.r.t. $\mathcal{P}' = \left\{A' + \lambda B \right\}$.\end{corollary}

\section{Poisson pencils}

Two Poisson brackets $\mathcal{A}$ and $\mathcal{B}$ are \textbf{compatible} if any their linear combination $\alpha \mathcal{A} + \beta \mathcal{B}$ with constant coefficients is also a Poisson bracket. In local coordinates $x^i$ this condition can be written as  \begin{equation} \label{Eq:CompBrack} \sum_{\operatorname{cyc}(i, j, k)} \sum_{s} \left(\mathcal{A}^{is}  \frac{\partial \mathcal{B}^{jk}}{\partial x^s} +\mathcal{B}^{is}  \frac{\partial \mathcal{A}^{jk}}{\partial x^s} \right) =0,\end{equation} where $\sum_{\operatorname{cyc}(i, j, k)}$ denotes the cyclic sum over the indices $i, j$ and $k$. It is well-known (see e.g. \cite{DufourZung05}) that Poisson brackets  $\mathcal{A}, \mathcal{B}$  are compatible if and only if their Schouten--Nijenhuis bracket vanishes $[\mathcal{A}, \mathcal{B} ] = 0$.

\begin{assertion} \label{A:BrackComCas} Assume that in coordinates $\left(x^1,\dots, x^{n}, z^{n+1},\dots, z^{n+m}\right)$ a Poisson pencil $\mathcal{A}_{\lambda} = \mathcal{A}+\lambda \mathcal{B}$ has the form \begin{equation} \label{Eq:BrackComCas} \mathcal{A}_{\lambda} = \left( \begin{matrix} \hat{\mathcal{A}_{\lambda}} (x, z) & 0_{n\times m} \\ 0_{m \times n} & 0_{m\times m}\end{matrix} \right). \end{equation} The Poisson brackets $\mathcal{A}$ and $\mathcal{B}$ are compatible if and only if the corresponding brackets $\hat{\mathcal{A}} (x, z)$ and $\hat{\mathcal{B}} (x, z)$ are compatible for all fixed values of $z$. \end{assertion}

\begin{proof}[Proof of Assertion~\ref{A:BrackComCas} ] Condition~\eqref{Eq:CompBrack} holds for the brackets \eqref{Eq:BrackComCas} if and only if it holds for the indices $i, j, k$ and $s$  ranging from $1$ to $n$. Assertion~\ref{A:BrackComCas} is proved. \end{proof}

In \cite{BolsZhang} the \textbf{characteristic polynomial} $p_{\mathcal{P}}(\lambda)$ of $\mathcal{P} = \left\{ \mathcal{A} + \lambda \mathcal{B}\right\}$ is defined as follows. Consider all diagonal minors $\Delta_I$ of the matrix $\mathcal{A} + \lambda \mathcal{B}$ of order rank $\mathcal{P}$ and take the Pfaffians $\operatorname{Pf}(\Delta_I)$, i.e. square roots, for each of them. The characteristic polynomial is the greatest common divisor of all these Pffaffians: \[ p_{\mathcal{P}} = \operatorname{gcd} \left(\operatorname{Pf}(\Delta_I) \right). \]

\subsection{Constructing new Poisson pencils using Casimir functions}

A function $f$ is a \textbf{Casimir function} of a Poisson bracket $\mathcal{A}$ if $\mathcal{A} df = 0$. We denote the set of all Casimir functions associated with a Poisson bracket $\mathcal{A}$ as $\mathcal{C}\left( \mathcal{A}\right)$.

\begin{assertion}[{\cite[Assertion 3.1]{Kozlov24BiIntBiPoissReduct}}] \label{A:NewPoissonCasimir} Let $\mathcal{A}$ and $\mathcal{B}$ be two compatible Poisson brackets on $M$. Assume that $f$ is a Casimir function for both brackets, i.e. $f \in \mathcal{C} \left( \mathcal{A} \right) \cap \mathcal{C}\left(\mathcal{B}\right)$. Then we have the following:

\begin{enumerate}
    \item The sum $\mathcal{A}_f = \mathcal{A} + f \mathcal{B}$ is a well-defined Poisson bracket on $M$.
    
    \item The bracket $\mathcal{A}_f$ is compatible with the brackets $\mathcal{A}$ and $\mathcal{B}$.

    \item The KCF of $\mathcal{A}_f(x) + \lambda \mathcal{B}(x)$ can be obtained from KCF of $\mathcal{A}(x) + \lambda \mathcal{B}(x)$ if we replace each eigenvalue $\lambda_j(x)$ with $\lambda_j(x) + f(x)$.

    \item Functions $g$ and $h$ are in bi-involution w.r.t. $\mathcal{A}$ and $\mathcal{B}$ if and only if they are in bi-involution w.r.t. $\mathcal{A}_f$ and $\mathcal{B}$.
\end{enumerate} 

\end{assertion}

\subsection{Poisson pencils with common Casimirs}  \label{S:LocTur3Flat}

\begin{definition} 
A Poisson pencil $\mathcal{P} = \left\{ \mathcal{A}_{\lambda}\right\}$ on $M$ is \textbf{flat} if for any point $x_0 \in M$ there exist local coordinates $x^1, \dots, x^n$ such that all Poisson structures $\mathcal{A}_{\lambda}$ have constant coefficients: \[ \mathcal{A}_{\lambda} = \sum_{i < j} \left( c_{ij} + \lambda d_{ij} \right) \frac{\partial}{\partial x^i} \wedge \frac{\partial}{\partial x^j}, \qquad c_{ij} = \operatorname{const}, \quad d_{ij} = \operatorname{const}.\]
\end{definition} 

Let $\mathcal{P} = \left\{ \mathcal{A}_{\lambda} = \mathcal{A} + \lambda \mathcal{B}\right\}$ be a Poisson pencil on real $C^{\infty}$-smooth manifold $M$ and $p_{\mathcal{P}(x)}$ be its characteristic polynomial at $x\in M$.

\begin{definition} We call a Poisson pencil $\mathcal{P}$ a \textbf{pencil with common Casimirs} if for all $x \in M$ we have \[\deg p_{\mathcal{P}(x)} =  \operatorname{rk} \mathcal{P}(x) = \operatorname{const}.\] \end{definition} A Jordan-Kronecker decomposition\footnote{See \cite[Section 2]{Kozlov24BiIntBiPoissReduct} for the definition of JK decomposition.} of the Poisson pencil  $\mathcal{P}(x)$ consists of Jordan blocks and $r = \dim M - \operatorname{rk}\mathcal{P}$ trivial $1 \times 1$ Kronecker blocks.   By selecting common local Casimir functions $z_1, \dots, z_r$ and extending them to local coordinates $x_1, \dots, x_{n - r}, z_1, \dots, z_r$, the matrices of Poisson brackets assume the following form: \begin{equation} \label{Eq:LocCommCas} \mathcal{A}_{\lambda} = \left(\begin{matrix} \hat{\mathcal{A}}_{\lambda}(x,z) & 0 \\ 0 & 0 \end{matrix} \right),\end{equation}  where the brackets $\hat{\mathcal{A}}_{\lambda}$ are nondegenerate for regular $\lambda$. \cite[Theorem 2]{turiel} provides a simple sufficient condition for flatness of Poisson pencils with common Casimirs. We reformulate this theorem as follows. 

\begin{theorem}[F.\,J.~Turiel, \cite{turiel}] \label{Th:Turiel2} A Poisson pencil $\mathcal{P}$ with common Casimirs on $M$ and a single zero eigenvalue $\lambda(x) = 0$ is flat in a neighborhood of any JK-regular point $x \in M$.\end{theorem}

\section{Bi-Poisson reduction} \label{S:BiPoisRed}

 Bi-Poisson reduction is the fundamental technique that enables us to prove bi-integrability of bi-Hamiltonian systems. In Section~\ref{SubS:LineRed} we present a linear analogue of bi-Poisson reduction for linear pencils. The main result is Theorem~\ref{T:BiPoissRed} in Section~\ref{SubS:BiPoisRedTh}. In Section~\ref{SubS:Split}, we will demonstrate that after performing bi-Poisson reduction, it is possible to ``split'' the Poisson pencil according to its eigenvalues.

\subsection{Linear bi-Poisson reduction} \label{SubS:LineRed}

The next theorem is an analogue of linear symplectic reduction for a pair of 2-forms.

\begin{theorem} \label{T:BiPoissReduction} Let $\mathcal{L} = \left\{A_{\lambda} \right\}$ be a linear pencil on $V$ and let $U\subset \left(V, \mathcal{L} \right)$ be an admissible bi-isotropic subspace. Then

\begin{enumerate}

\item The induced pencil $\mathcal{L} ' = \left\{A'_{\lambda}\right\}$ on $U^{\perp}/ U$ is well-defined. 

\item If $L$ is a bi-Lagrangian (or bi-isotropic) subspace of $(V, B)$, then \[ L' = \left( \left( L \cap U^{\perp}\right) + U \right) / U\] is a bi-Lagrangian (respectively, bi-isotropic) subspace of $U^{\perp}/U$. 

\end{enumerate}

\end{theorem}

We need the following simple statement.

\begin{assertion}[{\cite[Assertion 4.1]{Kozlov24BiIntBiPoissReduct}}] \label{A:AdmSpectr} Under the conditions of Theorem~\ref{T:BiPoissReduction}, if the admissible subspace contains the core subspace $K \subset U$, then the following holds.

\begin{enumerate} 

\item All eigenvalues of $\mathcal{L} '$ are eigenvalues of $\mathcal{L} $, i.e. \begin{equation} \label{Eq:SpectrSub} \sigma(\mathcal{L} ') \subseteq \sigma(\mathcal{L} ).\end{equation} In other words, if $A_{\lambda} \in \mathcal{L} $ is regular, then the induced form $A'_{\lambda}$ is also regular.

\item The induced pencil  $\mathcal{L} ' = \left\{A'_{\lambda}\right\}$ is nondegenerate, i.e. $\operatorname{Ker}\mathcal{A}'_{\lambda} = 0$ for generic $\lambda$.

\end{enumerate}

\end{assertion}

\subsection{Bi-Poisson reduction theorem} \label{SubS:BiPoisRedTh}

The next result is the main technique that allows us to bi-integrate bi-Hamiltonian systems. A subbundle $\Delta \subset T^*M$ is bi-isotropic (admissible, etc) if each subspace $\Delta_x \subset T^*_x M$ is bi-isotropic (admissible, etc).

\begin{theorem}[{\cite[Theorem 4.3]{Kozlov24BiIntBiPoissReduct}}]  \label{T:BiPoissRed}  Let $\mathcal{P} = \left\{ A_{\lambda} = \mathcal{A} + \lambda \mathcal{B}\right\}$ be a Poisson pencil on $M$ such that $\operatorname{rk} \mathcal{\mathcal{P}}(x) = 2k$ for all $x \in M$. Let $\Delta \subset T^*M$ be an integrable bi-isotropic admissible subbundle that contains the core distribution $\mathcal{K} \subset \Delta$. Then the following holds:

\begin{enumerate} 

\item $\Delta^{\perp}$ is an integrable admissible subbundle of  $T^*M$. 

\item Moreover, there exist local coordinates \begin{equation} \label{Eq:LocCoorBiPoissRed} (p, f, q) = (p_1,\dots, p_{m_1}, f_1, \dots, f_{m_2}, q_1, \dots, q_{m_3})\end{equation} such that \begin{equation} \label{Eq:DDperpleDistLoc} \Delta = \operatorname{span}\left\{dq_1, \dots, dq_{m_3} \right\}, \quad \Delta^{\perp} = \operatorname{span}\left\{df_1, \dots, df_{m_2}, dq_1, \dots, dq_{m_3} \right\}\end{equation} and the pencil has the form \begin{equation} \label{Eq:BiPoissMat} \mathcal{A}_{\lambda} = \sum_{i=1}^{m_1} \frac{\partial}{\partial p_i} \wedge v_{\lambda, i} + \sum_{1 \leq i < j \leq m_2} c_{\lambda, ij}(f, q)  \frac{\partial}{\partial f_i} \wedge \frac{\partial}{\partial f_j} \end{equation} for some vectors $v_{\lambda, i} = v_{\lambda, i}(p, f, q)$ and some functions $c_{\lambda, ij}(f, q)$. 

\end{enumerate}

\end{theorem}

Simply speaking, the matrices of the Poisson brackets in Theorem~\ref{T:BiPoissRed} take the form \[\mathcal{A}_{\lambda} = \left( \begin{matrix} * & * & * \\ * & C_{\lambda}(f, q) & 0 \\ * & 0 & 0\end{matrix} \right), \]  where $*$ are some matrices.  Obviously, the vector fields $v_{\lambda, i} = v_{\lambda, i}(x, s, y)$ and the functions $c_{\lambda, ij}(s, y)$ depend linearly on $\lambda$: \[ v_{\lambda, i} = v_{0, i} + \lambda v_{\infty, i}, \qquad c_{\lambda, ij}(f, q) = c_{0, ij}(f, q) + \lambda c_{\infty, ij}(f, q).\]

\begin{definition} Let $\mathcal{P}$ be a Poisson pencil on $M$ with constant rank and $\Delta \subset T^*M$ be an integrable bi-isotropic admissible subbundle. We perform a local \textbf{bi-Poisson reduction} near $x \in M$ by quotienting a sufficiently small neighborhood $U$ of $x$ by the distribution $\left(\Delta^{\perp}\right)^0$. This induces a new Poisson pencil $\mathcal{P}'$ on the quotient space $U/ \left(\Delta^{\perp}\right)^0$,  with the projection \[ \pi: (U, \mathcal{P}) \to \left(U/ \left(\Delta^{\perp}\right)^0, \mathcal{P}' \right).\] \end{definition}

Theorem~\ref{T:BiPoissRed} guarantees that we can perform (local) bi-Poisson reduction. In the local coordinates $(p, f, q)$ from this theorem \[ \left(\Delta^\perp\right)^0 = \operatorname{span}\left\{\frac{\partial}{\partial p_1}, \dots, \frac{\partial}{\partial p_{m_1}}\right\}.\] Thus, $(f,q)$ are local coordinates on the quotient $U/ \left(\Delta^{\perp}\right)^0$ and the induced pencil $\mathcal{P}'$ takes the form \begin{equation} \label{Eq:InducedPencil} \mathcal{P}' = \left(\begin{matrix}  C_{\lambda}(f, q) & 0 \\ 0 & 0 \end{matrix} \right).\end{equation}

\subsection{Factorization theorem} \label{SubS:Split}

Consider a pencil \eqref{Eq:InducedPencil} induced after bi-Poisson reduction. We can ``group'' the coordinates $f$ by eigenvalues. Formally, we have the following statement. 

\begin{theorem}  \label{T:TrivKronFact} Let $\mathcal{P} = \left\{ \mathcal{A}_{\lambda} = \mathcal{A} + \lambda \mathcal{B}\right\}$ be a Poisson pencil on a real smooth manifold $M$ and $p_{\mathcal{P}(x)}$ be its characteristic polynomial at $x\in M$. Assume the following:

\begin{enumerate}
    \item For all $x \in M$ we have \begin{equation} \label{Eq:LocRkPConst} \deg p_{\mathcal{P}(x)} =  \operatorname{rk} \mathcal{P}(x) = \operatorname{const}.\end{equation} 

\item At a point $p \in M$ the characteristic polynomial $p_{\mathcal{P}(x)}$  has $k$ real (distinct) eigenvalues $\lambda_1, \dots, \lambda_k$  with multiplicities $m_1, \dots, m_k$ respectively and $s$ pairs of complex (non-real) conjugate eigenvalues $\mu_1, \bar{\mu}_1,\dots, \mu_s, \bar{\mu}_s$ with  multiplicities $l_1, \dots, l_s$.
    
\end{enumerate}  Then in a neighborhood of $p \in M$ there exists a local coordinate
system
\begin{gather*} x_1 = \left(x_1^1, \dots, x_1^{2m_1}\right), \qquad \dots, \qquad x_k = \left(x_k^1, \dots, x_k^{2m_k}\right), \\ u_1 = \left(u_1^1, \dots, u_1^{4l_1}\right), \quad \dots, \quad u_s = \left(u_s^1, \dots, x_s^{4l_s}\right), \quad z = (z_1, \dots, z_r), \end{gather*} such that the matrices of Poisson brackets have the form \begin{equation} \label{Eq:FormOneEigen} \mathcal{A}_{\lambda} = \left( \begin{matrix} C^1_{\lambda}(x_1, z) & & & & & & & \\ & \ddots & & & & & \\ & & C^k_{\lambda}(x_k, z) & & & & & \\ & & & D^1_{\lambda}(u_1, z) & & & & \\ & & & & \ddots & & & \\ & & &  & & & D^s_{\lambda}(u_s, z)  & \\ & & &  & & &   & 0_r \end{matrix} \right). \end{equation} Moreover,  at the point $p \in M$ each characteristic polynomial of the pencils $\left\{ C^t_{\lambda}(x_t, z) \right\}$ has a single real eigenvalue. And each characteristic polynomial of the pencils $\left\{ D^t_{\lambda}(u_t, z) \right\}$ has a single pair of complex eigenvalues at $p \in M$. \end{theorem}

\begin{proof}[Proof of Theorem~\ref{T:TrivKronFact}] Since \eqref{Eq:LocRkPConst} holds, locally the pencil $\mathcal{P}$ has the form \eqref{Eq:LocCommCas}. Let  $z_1, \dots, z_r$ be common (local) Casimir functions of regular pencils $\mathcal{A}_{\lambda}$. On each common symplectic leaf $S_z = \left\{ z_1 = \operatorname{const}, \dots, z_r = \operatorname{const}\right\}$ the pencil $\mathcal{P}$ defines a nondegenerate\footnote{A pair of nondegerenerate Poisson brackets $\mathcal{A}$ and $\mathcal{B}$ are compatible iff the recursion operator $P =\mathcal{A}\mathcal{B}^{-1}$ is a Nijenhuis operator, i.e. $N_{P} = 0$. Compatible nondegenerate Poisson brackets are the same as compatible symplectic forms $\mathcal{A}^{-1}$ and $\mathcal{B}^{-1}$.} Poisson pencil $\mathcal{P}^z=\left\{\hat{\mathcal{A}}_{\lambda} \right\}
$. We can "split" the nondegenerate pencils $\mathcal{P}^z$ using \cite[Lemma 2]{turiel}. Alternatively, one can use  the splitting theorem for Nijenhuis operators (see \cite[Theorem 3.1]{BolsinovNijenhuis}). We get coordinates $x_1, \dots, x_k, u_1, \dots u_s$ such that the matrices of the pencils $\mathcal{P}^z$ are block-diagonal: \[ \mathcal{P}^z =  \left( \begin{matrix} C^1_{\lambda}(x_1, z) & &  \\ & \ddots &  \\ & & D^s_{\lambda}(u_s, z) \end{matrix} \right). \] Since $z_i$ are Casimir function, the pencil $\mathcal{P}$ takes the form \eqref{Eq:FormOneEigen}. Theorem~\ref{T:TrivKronFact} is proved. \end{proof}

\section{Standard integrals} \label{SubS:StandInt}

Standard integrals were defined in Definition~\ref{Def:StandInt}. The next statement is proved similar to \cite[Lemma 5.1, Lemma 5.3]{Kozlov24BiIntBiPoissReduct}.

\begin{lemma} \label{L:StandIntAdm} Let $v = \mathcal{A}_{\lambda}d H_{\lambda}$ be a system that is bi-Hamiltonian w.r.t. a pencil $\mathcal{P}= \left\{\mathcal{A}_{\lambda}\right\} $. The family of standard integrals $\mathcal{F}$ on $M$ is an admissible family of functions in bi-involution. The standard integrals $\mathcal{F}$ on $M$ are first integrals of the bi-Hamiltonian system. \end{lemma}

 In the proof of Lemma~\ref{L:StandIntAdm} we should replace the following statement (which is \cite[Lemma 5.2]{Kozlov24BiIntBiPoissReduct}) with its analog  Lemma~\ref{L:EigenLemmaRealConj} for complex conjugate eigenvalues.

\begin{lemma} \label{L:EigenDiff} Let $\mathcal{P} = \left\{ \mathcal{A} + \lambda \mathcal{B} \right\}$ be a Poisson pencil on a manifold $M$. For any JK-regular point $x \in (M, \mathcal{P})$ and any finite eigenvalue $\lambda_i(x) < \infty$ we have \begin{equation} \label{Eq:Eigen1} (\mathcal{A} - \lambda_i (x) \mathcal{B}) d\lambda_i (x) =0.\end{equation} \end{lemma}

The next statement follows from Theorem~\ref{T:TrivKronFact}.

\begin{lemma} \label{L:EigenLemmaRealConj} Let $\mathcal{P} = \left\{\mathcal{A} + \lambda \mathcal{B} \right\}$ be a Poisson pencil on real manifold $M$ and $\lambda(x) = \alpha(x) + i \beta(x)$ be its complex eigenvalue on $M$. Then almost everywhere on $M$  we have \begin{equation} \label{Eq:DLambdaMainRealConj} d\lambda(x) = d\alpha(x) + i \cdot d \beta(x) \in \operatorname{Ker}_{-\lambda(x)}^{\mathbb{C}} + \mathcal{K}^{\mathbb{C}}.  \end{equation} 
\end{lemma}

Here at each point $x \in M$ we complexify the cotangent space $T^* M$ and extend $\mathcal{A}(x)$ and $\mathcal{B}(x)$ to the skew-symmetric forms $\mathcal{A}^{\mathbb{C}} (x)$ and $\mathcal{B}^{\mathbb{C}} (x)$ on $(T^*M)^{\mathbb{C}}$. Then $\mathcal{K}^{\mathbb{C}}$ is the complexification of the core distribution $\mathcal{K}$ and  \[\operatorname{Ker}_{-\lambda(x)}^{\mathbb{C}} = \operatorname{Ker}\left(\mathcal{A}^{\mathbb{C}}(x) - \lambda(x) \mathcal{B}^{\mathbb{C}}(x)\right). \]

\section{Proof of Theorem~\ref{T:MainGenTh}} \label{S:MainProof}

The proof is in several steps:

\begin{enumerate}

\item \textit{Perform bi-Poisson reduction w.r.t. $d\mathcal{F}$} (see Section~\ref{S:BiPoisRed}). Note that $d \mathcal{F}$ satisfies conditions of Theorem~\ref{T:BiPoissRed} by Lemma~\ref{L:StandIntAdm} (also, $\mathcal{K} \subseteq d \mathcal{F}$ by Definition~\ref{Def:SmallNeigh}). The next statement easily follows from Theorem~\ref{T:BiPoissRed} and Assertion~\ref{A:AdmSpectr}.

\begin{assertion} \label{A:AfterRed} Let $\mathcal{P}'$ denote the Poisson pencil that results from performing bi-Poisson reduction w.r.t. the family of standard integrals $\mathcal{F}$. The following properties hold:

\begin{enumerate}

\item $\mathcal{P}'$ is a Poisson pencil with common Casimirs. In other words, in some local coordinates  $(f, q) = (f_1, \dots, f_{n_1}, q_1, \dots, q_{n_2})$ the pencil has the form \begin{equation} \label{Eq:AfterRedPenc} \mathcal{P}' = \left( \begin{matrix} C_{\lambda}(f, q) & 0 \\ 0 & 0 \end{matrix}\right), \end{equation}  where the pencils $\left\{C_{\lambda}(f, q) \right\}$ are nondegenerate for fixed $z$.

\item All eigenvalues of the pencil $\mathcal{P}'$ are its common Casimir functions. For complex eigenvalues $\alpha_j(x) + i \beta_j(x)$ the real part $\alpha_j(x)$ and the imaginary part $\beta_j(x)$ are common Casimir functions.

\item All eigenvalues of $\mathcal{P}'$ are also eigenvalues of the original pencil $\mathcal{P}$.

\end{enumerate}

\end{assertion} 

As we proceed, we substitute the pencil $\mathcal{P}$ with the reduced pencil $\mathcal{P}'$.

\item \label{Step:RedEigen} \textit{Reduction to the case of one real eigenvalue or a pair of complex conjugate complex eigenvalues}. After the bi-Poisson reductions, the pencil $\mathcal{P}$ has the form \eqref{Eq:AfterRedPenc}. Hence, we can use Theorem~\ref{T:TrivKronFact}.  By selecting the appropriate coordinates $x_j, z$ or $u_j, z$ we can effectively reduce the general case to a simpler scenario involving either a single eigenvalue  $\lambda_j(x)$ or a pair of complex conjugate eigenvalues $\alpha_j(x) \pm i \beta_j(x)$.

\item \textit{Case of a single real eigenvalue}. Assume that after the previous step there is only one eigenvalue $\lambda_1(x)$ and it is a common Casimir function for all brackets \[ \lambda_1(x) \in \mathcal{C}\left( \mathcal{A}_{\lambda}\right), \qquad \forall \lambda \in \bar{\mathbb{C}}.\] Without the loss of generality, $\lambda_1(x) = 0$, since by Assertion~\ref{A:NewPoissonCasimir} we can replace the pencil $\mathcal{P}$ with the new pencil \[\hat{\mathcal{P}} = \left\{\hat{\mathcal{A}} + \lambda \mathcal{B} \right\}, \qquad \hat{\mathcal{A}} = \mathcal{A} - \lambda_1(x) \mathcal{B}.\] By Theorem~\ref{Th:Turiel2} the considered Poisson pencil $\left\{\mathcal{A} + \lambda \mathcal{B} \right\}$ is locally flat. Thus, we can take local coordinates in which the pencil has constant coefficients \[\mathcal{A}^{ij}(x) = \operatorname{const}, \qquad \mathcal{B}^{ij}(x) = \operatorname{const}.\] The family of functions $\mathcal{F}$ can be easily extended to a complete family of functions $\mathcal{G}$ that are in bi-involution. This can be accomplished by incorporating the coordinate functions $g_1,\dots, g_N$ that define the bi-Lagrangian subspace $L$.

\item \textit{Case of a pair of complex conjugate complex eigenvalues}. Assume that after Step~\ref{Step:RedEigen} there is a pair of complex conjugate eigenvalues $\alpha(x) \pm i \beta(x)$. Without loss of generality, the pencil $\mathcal{P}$ has the form \eqref{Eq:AfterRedPenc}. Complexify each symplectic leaf $\left\{q_j = \operatorname{const}\right\}$ similarly to how it is done in \cite[Section 6]{turiel} or how it is done for Nijenhuis operators in \cite[Section 3.3]{BolsinovNijenhuis}. In short, the semi-simple part of the recursion operator defines a complex structure. Take the function \[ f(z) = \begin{cases} i, \quad \operatorname{Im} z > 0, \\ -i, \quad \operatorname{Im} z < 0. \end{cases} \] By \cite[Proposition 3.2]{BolsinovNijenhuis}  $J = f(P)$, where $P$ is the recursion operator, is a complex structure $J^2 = - \operatorname{id}$. After complexification on each symplectic leaf $\left\{q_j = \operatorname{const}\right\}$ we get a nondegenerate Poisson pencil with one complex eigenvalue $\lambda_0(x) = \alpha(x) + i \beta(x)$.  We use the following statement.

\begin{assertion} \label{A:NewBrack} The Poisson bracket \[ \hat{\mathcal{A}} = \mathcal{A} - \lambda_0(x) \mathcal{B} = \mathcal{A} - \alpha(x) \mathcal{B} - \beta(x) J \circ \mathcal{B}\] is compatible with $\mathcal{B}$. The pencil $\hat{\mathcal{P}} = \left\{\hat{\mathcal{A}} + \lambda \mathcal{B}\right\}$ has one zero eigenvalue and, therefore, it is flat.  \end{assertion}

\begin{proof}[Proof of Assertion~\ref{A:NewBrack}] By Assertion~\ref{A:BrackComCas}, it suffices to prove that $\hat{\mathcal{A}}$ and $\mathcal{B}$ are compatible on each symplectic leaf $S_q = \left\{q_j = \operatorname{const}\right\}$. By Assertion~\ref{A:AfterRed}, the functions $\alpha(x)$ and $\beta(x)$ are constants on each symplectic leaf $S_q$. Hence, on each leaf the complexified $\hat{\mathcal{A}} = \mathcal{A} - \lambda_0(x) \mathcal{B}$ is a linear combination with constant coefficients and it is compatible with $\mathcal{B}$. The pencil $\hat{\mathcal{P}}$ is flat by Theorem~\ref{Th:Turiel2}. Assertion~\ref{A:NewBrack} is proved.  \end{proof}

 By Corollary~\ref{Cor:BiLagrNilp} the bi-Lagrangian subspace $L$ is also bi-Lagrangian for $\hat{\mathcal{P}}$. In coordinates where $\hat{\mathcal{P}}$ has constant coefficients, we construct a complete family of functions $\mathcal{G}$ for both $\hat{\mathcal{P}}$ and $\mathcal{P}$ by selecting coordinate functions $g_1,\dots, g_N$, similar to the case of one real eigenvalue.
    
\end{enumerate}

Theorem~\ref{T:MainGenTh} is proved.

\begin{remark} Note that after bi-Poisson reduction w.r.t. standard integrals the induced pencil $\mathcal{P}'$ becomes flat when restricted to each common symplectic leaf (i.e. $\mathcal{P}'$ is \textbf{leaf-wise flat}). By applying Turiel's local coordinates from  \cite{turiel} to each leaf, we can locally bring $\mathcal{P}'$ to the form as in the Jordan-Kronecker theorem, but with eigenvalues $\lambda(z)$ depending on the common Casimir functions.  \end{remark}

\end{document}